\documentclass{article}
\usepackage[left=1in, top=1in, right=1in, bottom=1in]{geometry}
\usepackage{amsmath,amssymb,amsthm}
\usepackage{enumerate}
\numberwithin{equation}{section}
\newtheorem{theorem}{Theorem}[section]
\newtheorem{lemma}{Lemma}[section]
\newtheorem{remark}{Remark}[section]
\newtheorem{proposition}{Proposition}[section]
\newtheorem{corollary}{Corollary}[section]

\usepackage{graphicx} 
\usepackage{xcolor}

\usepackage{hyperref}
\usepackage{pgfplots}
\usepackage{pgfplotstable}
\pgfplotsset{compat=1.18}
\def\DimN{4}       

\definecolor{mult}{RGB}{190,55,45}
\definecolor{ours}{RGB}{35,100,185}
\definecolor{tanaka}{RGB}{20,140,85}

\pgfplotsset{
  regionaxis/.style={
    width=14cm,
    height=9cm,
    xmin=2,
    ymin=0,
    ymax=10,
    xlabel={$p$},
    ylabel={$\alpha$},
    axis lines=box,
    grid=major,
    grid style={gray!15},
    tick label style={font=\small},
    label style={font=\large},
    title style={font=\normalsize,align=center},
    scaled ticks=false,
    enlargelimits=false,
    clip=true,
    legend style={
      at={(0.5,-0.18)},
      anchor=north,
      draw=none,
      font=\small,
      cells={anchor=west}
    }
  }
}

\pgfplotstableread{
a       p
0       2.665908288
0.025   2.650934301
0.05    2.636525084
0.075   2.622650173
0.1     2.609281245
0.15    2.583957684
0.2     2.560364121
0.25    2.538334484
0.3     2.517723175
0.35    2.498402004
0.4     2.480257663
0.5     2.447108404
0.6     2.417594973
0.7     2.391171563
0.8     2.367395315
0.9     2.345903201
1       2.326394864
1.25    2.284754799
1.5     2.251087295
1.75    2.223395078
2       2.200287092
2.5     2.164099050
3       2.137257901
3.5     2.116718050
4       2.100603368
4.5     2.087699738
5       2.077189618
5.5     2.068504182
6       2.061236634
6.5     2.055089157
7       2.049839270
7.5     2.045317815
8       2.041394166
8.5     2.037966049
9       2.034952400
9.5     2.032288246
10      2.029920988
}\OurBoundary

\pgfplotstableread{
a             p
0             2.261815732
0.015008757   2.246033943
0.030017513   2.230588126
0.045026270   2.215467973
0.060035026   2.200663590
0.075043783   2.186165483
0.090052539   2.171964531
0.105061296   2.158051976
0.120070052   2.144419395
0.135078809   2.131058691
0.150087565   2.117962075
0.165096322   2.105122049
0.180105078   2.092531398
0.195113835   2.080183168
0.210122591   2.068070664
0.225131348   2.056187429
0.240140104   2.044527240
0.255148861   2.033084094
0.270157617   2.021852200
0.285166374   2.010825967
0.300175130   2.000000000
}\TanakaBoundary

\newcommand{\DrawUniquenessOneD}{%
  \addplot[
    draw=none,
    fill=ours!20,
    forget plot
  ] table[x=p,y=a] {\OurBoundary}
    -- (axis cs:2,10)
    -- (axis cs:2,0)
    -- cycle;

  \addplot[
    draw=none,
    fill=tanaka!40,
    forget plot
  ] table[x=p,y=a] {\TanakaBoundary}
    -- (axis cs:2,0)
    -- cycle;

  \addplot[
    ours,very thick,no marks,forget plot
  ] table[x=p,y=a] {\OurBoundary};

  \addplot[
    tanaka!80!black,very thick,no marks,forget plot
  ] table[x=p,y=a] {\TanakaBoundary};
}

\title{Sufficient conditions for multiplicity and uniqueness of positive solutions of the H\'enon equation}
 \date{}
\author{ Geyang Du\\
\small gydu@jiangnan.edu.cn\\
\small School of mathematics and data science, Jiangnan University, Wuxi, People’s
Republic of China}
\begin{document}
\maketitle
\section*{Abstract}
We establish explicit sufficient conditions for multiplicity of positive solutions of the Hénon equation in the unit ball
and for uniqueness in one dimension. In the Sobolev-subcritical range, a spherical harmonic function argument shows that the
radial solution has Morse index greater than one and therefore differs from a ground state. In even dimensions,
a degree-four perturbation in a symmetry and monotonicity cone yields a sufficient condition for nonradiality in
the supercritical range. The one-dimensional uniqueness criterion follows from a comparison of weighted Dirichlet
eigenvalues and a concavity estimate.

\section{Introduction and main results}
We study the following H\'enon equation:
\begin{equation}\label{Henon}
\left\{\begin{aligned}
-\Delta u&=|x|^{\alpha} u^{p-1},&&\text{in}\ B_1(0),\\
u&>0,&& \text{in}\ B_1(0),\\
u&=0,&& \text{on}\ \partial B_1(0),
\end{aligned}   \right. 
\end{equation}
where $B_1(0)\subset \mathbb{R}^N,\alpha>0,2<p<p_\alpha, p_\alpha=
\infty$ when $N=1,2$, $p_\alpha=\frac{2(N+\alpha)}{N-2}$ when $N\ge 3$.

H\'enon equation was intrduced in \cite{henon-AA1973} when astrophysicist Michel H\'enon was studying spherical stellar systems. Later, Ni proved the existence of radial solutions when $2<p<p_\alpha$ in \cite{Ni-Indiana1982}. The uniqueness of the radial solution was proved in \cite{Ni-Nussbaum-CPAM1985}, see also \cite{Gladiali-Grossi-Neves-Adv2013,Figueroa-Neves-Sergio-ANS2019}. Through Pohozaev identity, we know there is no solution when $p\ge p_\alpha$ \cite{Nagasaki-JFSUTS1989}. 

The interesting part of H\'enon equation is that it may have multiple solutions even though the domain $B_1(0)$ and operator $\Delta$ are sphercially invariant. This was first observed numerically by \cite{chen-zhou-ni-IJBCASE2000}. Because of the uniqueness of radial solution, multiplicity implies existence of nonradial solutions. This phenomena is called symmetry breaking, see \cite{Smoller-Wasserman-ARMA1986,Smoller-Wasserman-CMP1986,Smoller-Wasserman-Invent1990}. So H\'enon equation is very different from the famous Lane-Emden equation when $\alpha=0$, which has only radial solution by using the famous moving plane method \cite{Gidas-Ni-Nirenberg-cmp1979}.

We'll introduce some state-of-art literature about the multiplicity and uniqueness results which are highly related with our paper.

We first concentrate on the case $N=1$. Tanaka proved when $(p-1)2^{p-1}[B(\alpha+1,p+1)]^{-1}\le \pi^2$, the one-dimensional H\'enon equation has a unique positive solution which is even, while it has at least three solutions when $(p-2)\alpha\ge 4$.

When $N\ge2$, roughly speaking, uniqueness holds when $p-2,\alpha$ are small while multiplicity holds when $p-2,\alpha$ are large.

We first list some uniqueness results. Let $2^*=\frac{2N}{N-2}$ when $N\ge 3$, $2^*=\infty$ when $N=1,2$. Kajikiya proved when $2<p<2^*,\alpha$ is small, the ground state solution of H\'enon equation is unique, see \cite{Kajikiya-JDE2012}. When $p$ is close to $2$, Amadori and Gladiali proved the uniqueness in \cite{Amadori-Gladiali-ADE2014}.

As for multiplcity, the situation is a bit complicated. We seperate the results into subcritical $2<p<2^*$, critical $p=2^*$ and supercritical case $2^*<p<p_\alpha$. The first subcritical result was obtained in \cite{Smets-Willem-Su-ccm2002}. They compared the asymptotic behavior of the energy in the space $H_0^1(B_1(0))$ and radial functions space $H_R$, and proved the multiplicity when $2<p<2^*$ and $\alpha$ is large. Later, Serra extend their results into critical case $p=2^*$ when $N\ge 4$ and $\alpha$ is large, see \cite{Serra-CVPDE2005}. Yan and Wei remove the '$\alpha$ large' condition, and they also prove there are infinite many solutions when $p=2^*, N\ge 4$, see \cite{wei-yan-RMI2013}. The case $N=3$ was proved in \cite{Hao-Chen-Zhang-JDE2015}. Amadori and Gladali obtained the multiplicity results for $\alpha>0, N=2$ when $p$ is large in \cite{Amadori-Gladiali-JDE2020}. Recently, Liu and Luo construct infinite nonradial sign-changing solutions in the critical case\cite{Liu-Luo-MA2026}. 

As for the supercritical case, the results are incomplete. We list some progress toward this direction. Badiale and Serra construct some nonradial solutions when $p<p_N=2\frac{\lceil N/2 \rceil+1}{\lceil N/2\rceil-1}$ and $\alpha$ is large. They prove this by studying the asymptotic behavior of the energy in the double resolution space $H_{m,n}$ and $H_R$, see \cite{Badiale-Serra-ANS2004}. Figueroa and Neves proved the existence of non-raidal solutions when $\alpha$ is close to even number and $p$ is colse to $p_\alpha$, see \cite{Figueroa-Neves-Sergio-ANS2019}. Lately, Cowan and Moameni obtain the multiplicity results when $\frac{16(N+2)}{(N-2)^2}+2=\tilde p_N<p<p_\alpha$ and $N$ is even \cite{Cowan-Moameni-MA2024}.
It's not difficult to see that when $N$ is even, $p_N<\tilde p_N$. So we still don't know whether H\'enon equation has multiple solutions in the whole supercritical case.

Bifurcation results and the calculation of morse index of the H\'enon equation could be found in \cite{Amadori-CPAA2020,Amadori-DeMarchis-Ianni-NA2022,Amadori-Gladiali-NARWA2020,Amadori-Gladiali-Nonlinearity2020} and references therein.

In this article, we will prove the multiplicity results under certain conditions in the whole supercritical case. In addition, We notice, the seperation about the uniqueness and multiplicity region in $N\ge2$ are obtained numerically in \cite{Smets-Willem-Su-ccm2002}. We will provide some analytic evidence. Here are our main results.

\begin{theorem}[Multiplicty]\label{mainresultmultiplicity}
H\'enon equation \eqref{Henon} has at least two solutions when
\begin{itemize}
    \item $(p-2)(\alpha+N-2)\ge 4,p<2^*,N\ge 2$,
    \item $(p-2)(\alpha+N-8)\ge 16,2^*<p<p_\alpha,N\ge 3$ and $N$ is even.
\end{itemize} 
\end{theorem}
By theorem \ref{mainresultmultiplicity}, we immediately obtain the following corrollary.
\begin{corollary}\label{wholesupercritical}
When $\alpha\ge 3N, N\ge 3, N$ is even, the H\'enon equation \eqref{Henon} has at least two solutions for any $2^*<p<p_\alpha$.    
\end{corollary}
To the best of our knowledge, this is the first explicit sufficient condition on $\alpha$ that guarantees symmetry breaking for every $p\in(2^*,p_\alpha)$ in even dimensions $N\ge4$. What's more, multiplicity region in \cite{Cowan-Moameni-MA2024} is contained in ours. In fact, when $\tilde p_N<p<p_\alpha$, $\alpha>\frac{p}{2}(N-2)-N>8\frac{N+2}{N-2}-2$. So $(p-2)(\alpha+N-8)>16\frac{N+2}{(N-2)^2}\left(\frac{32}{N-2}+N-2\right)>16$. For example, when $N=4, \alpha=23$, Cowan and Moameni obtained the multiplicity result for $26<p<27$ while our corollary 1.1 implies the multiplicity result for $4<p<27$. Our second result is about the uniqueness for one-dimensional H\'enon equation.

\definecolor{supermult}{RGB}{165,85,185}

\pgfmathsetmacro{\pcritical}{2+4/(\DimN-2)}
\pgfmathsetmacro{\pCross}{2+16/(\DimN-2)}

\pgfmathsetmacro{\pmaxTwo}{2+20/(\DimN-2)}
\pgfmathsetmacro{\alphaMaxTwo}{max(10,3*\DimN+1)}

\pgfmathsetmacro{\pentryTwo}
  {2+4/(\alphaMaxTwo+\DimN-2)}
\pgfmathsetmacro{\pAtTop}
  {2+2*(\alphaMaxTwo+2)/(\DimN-2)}

\pgfmathsetmacro{\subRegionP}{2+2/(\DimN-2)}
\pgfmathsetmacro{\subRegionA}{0.83*\alphaMaxTwo}

\pgfmathsetmacro{\subLabelP}{2+1.3/(\DimN-2)}
\pgfmathsetmacro{\subLabelA}
  {4/(\subLabelP-2)-(\DimN-2)}

\pgfmathsetmacro{\superRegionP}{2+13/(\DimN-2)}
\pgfmathsetmacro{\superRegionA}{0.91*\alphaMaxTwo}

\pgfmathsetmacro{\superLabelP}{2+10/(\DimN-2)}
\pgfmathsetmacro{\superLabelA}
  {8-\DimN+16/(\superLabelP-2)}

\pgfmathsetmacro{\weightedLabelP}{1+18/(\DimN-2)}
\pgfmathsetmacro{\weightedLabelA}{6.5}

\pgfmathsetmacro{\uniqueLabelP}{2+3/(\DimN-2)}
\pgfmathsetmacro{\uniqueLabelA}{0.53*\alphaMaxTwo}

\tikzset{
  curve label/.style={
    font=\tiny,
    fill=none,
    draw=none,
    inner sep=1.5pt
  },
  region label/.style={
    font=\scriptsize,
    align=center,
    fill=none,
    draw=none,
    inner sep=1pt
  }
}

\begin{figure}[htbp]\label{figure}
\centering

\begin{minipage}[t]{0.49\textwidth}
\vspace{0pt}
\centering
\begin{tikzpicture}
\begin{axis}[
  regionaxis,
  width=\linewidth,
  height=6.7cm,
  xmax=6,
  ymax=10,
  xtick={2,3,4,5,6},
  ytick={0,2,4,6,8,10},
  tick label style={font=\scriptsize},
  label style={font=\small},
  title={ 1 Dimension},
  title style={font=\small,align=center}
]

\addplot[
  draw=none,
  fill=mult!20,
  domain=2.4:6,
  samples=220,
  forget plot
] {4/(x-2)}
  -- (axis cs:6,10)
  -- cycle;

\DrawUniquenessOneD

\addplot[
  mult,very thick,
  domain=2.4:6,
  samples=220,
  forget plot
] {4/(x-2)};

\node[
  region label,
  text=mult!80!black
] at (axis cs:4.65,5.0)
  {Multiplicity\\region};

\node[
  curve label,
  anchor=south west,
  rotate=-33
] at (axis cs:3.03,3.95)
  {$(p-2)\alpha=4$};

\draw[ours,thin]
  (axis cs:2.12,0.65)
  -- (axis cs:3.15,1.2)
  node[
    region label,
    font=\tiny,
    anchor=west,
    text=ours!80!black
  ] {Our uniqueness\\region};

\draw[tanaka!80!black,thin]
  (axis cs:2.07,0.08)
  -- (axis cs:2.85,0.30)
  node[
    region label,
    font=\tiny,
    anchor=west,
    text=tanaka!60!black
  ] {Tanaka's uniqueness region};

\end{axis}
\end{tikzpicture}
\end{minipage}
\hfill
\begin{minipage}[t]{0.49\textwidth}
\vspace{0pt}
\centering
\begin{tikzpicture}
\begin{axis}[
  regionaxis,
  width=\linewidth,
  height=6.7cm,
  xmax=\pmaxTwo,
  ymax=\alphaMaxTwo,
  tick label style={font=\scriptsize},
  label style={font=\small},
  title={ \(N\) Dimension
         (\(N=\DimN\))},
  title style={font=\small,align=center}
]

\addplot[
  draw=none,
  fill=mult,
  fill opacity=0.20,
  domain=\pentryTwo:\pcritical,
  samples=200,
  forget plot
] {4/(x-2)-(\DimN-2)}
  -- (axis cs:\pcritical,\alphaMaxTwo)
  -- cycle;

%
%
%
\addplot[
  draw=none,
  fill=supermult,
  fill opacity=0.24,
  domain=6:\alphaMaxTwo,
  samples=240,
  forget plot
] (
  {max(\pcritical,2+16/(x+\DimN-8))},
  {x}
)
  -- (axis cs:\pAtTop,\alphaMaxTwo)
  -- (axis cs:\pCross,6)
  -- cycle;






\addplot[
  mult,thick,
  domain=\pentryTwo:\pcritical,
  samples=200,
  forget plot
] {4/(x-2)-(\DimN-2)};

\addplot[
  supermult,very thick,
  domain=\pcritical:\pCross,
  samples=220,
  forget plot
] {8-\DimN+16/(x-2)};

\addplot[
  black!65,dashed,thick,
  domain=\pcritical:\pmaxTwo,
  samples=2,
  forget plot
] {(\DimN-2)*(x-2)/2-2};

\addplot[
  black!65,dashed,thick,
  forget plot
] coordinates {
  (\pcritical,0)
  (\pcritical,\alphaMaxTwo)
};

\addplot[
  only marks,
  mark=o,
  mark size=2pt,
  mark options={draw=black,fill=none},
  forget plot
] coordinates {(\pCross,6)};

\node[
  region label,
  font=\tiny,
  text=mult!80!black
] at (axis cs:\subRegionP+0.2,\subRegionA)
  {Subcritical\\multiplicity\\region};

\node[
  region label,
  font=\tiny,
  text=supermult!80!black
] at (axis cs:\superRegionP,\superRegionA)
  {Supercritical\\multiplicity region};



  
\draw[ours,thin]
  (axis cs:3,2)
  -- (axis cs:\subLabelP+0.5,\subLabelA)
  node[
    region label,
    font=\tiny,
    anchor=west,
    text=ours!80!black
  ] {$(p-2)(\alpha+N-2)=4$};

\node[
  curve label,
  anchor=south,
  text=supermult!80!black,
  yshift=3pt
] at (axis cs:\superLabelP,\superLabelA)
  {$(p-2)(\alpha+N-8)=16$};

\node[
  curve label,
  anchor=north west,
  yshift=10pt
] at (axis cs:\weightedLabelP+0.5,\weightedLabelA)
  {$p=p_\alpha$};

\node[
  curve label,
  anchor=north east
] at (axis cs:\pcritical+.5,\alphaMaxTwo-3)
  {$p=2^*$};

\node[
  curve label,
  anchor=north east,
  yshift=-4pt
] at (axis cs:\pCross+1.5,6)
  {$\left(2+\frac{16}{N-2},\,6\right)$};

\end{axis}
\end{tikzpicture}
\end{minipage}

\caption{Multiplicity and uniqueness regions for H\'enon equation.}
\end{figure}
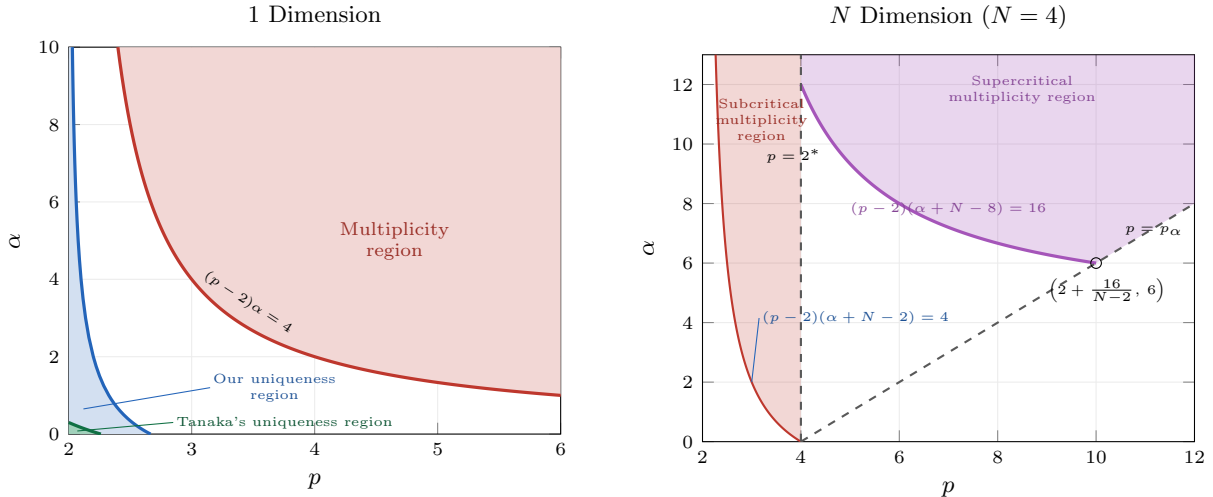

\begin{theorem}[Uniqueness]\label{mainresultuniqueness}
Set $\nu=\frac{1}{\alpha+2},c_p=(p-2)(p-1)^{-\frac{p-1}{p-2}}$, $j_{\nu,1}$ is the first positive zero of the Bessel function $J_\nu$. Let $D_\alpha=(\alpha+2)2^{-\nu}J_{-\nu+1}\left(j_{-\nu,1}\right)\left(j_{-\nu,1}\right)^{1+\nu}\Gamma\left(1-\nu\right)$, then
H\'enon equation \eqref{Henon} has a unique solution when 
\begin{align*}
 (p-1)\left(\frac{j_{-\nu,1}}{j_{\nu,1}}\right)^2+c_pD_\alpha\le 1,N=1. 
\end{align*}
\end{theorem}

The paper is organized as follows: in section \ref{gs}, we list all the notations needed and calculate the morse index of the ground state solution in the subcritical case, we also give a new proof of the existence of ground sate slution in the supercritical case; in section \ref{radial}, we calculate the morse index of the radial solution and prove theorem \ref{mainresultmultiplicity}; in section \ref{uniqueness} we prove theorem \ref{mainresultuniqueness}.

\section{Morse index of the ground state soluton}\label{gs}
We introduce weighted Lebesgue space $L^p(B_1(0),|x|^\alpha dx)$ with norm \[\|u\|_{p,\alpha}=\left(\int_{B_1(0)}   |x|^{\alpha}u^pdx\right)^{\frac{1}{p}}.\] 
\subsection{$N=2$ and subcritical case for $N\ge3$}\label{m(ugs)}
It is well known that $H_0^1(B_1(0))$ is compactly embedded into $L^p(B_1(0))$ for any $2<p<2^*$ when $N\ge2$. Because $L^p(B_1(0))$ is a subspace of $L^p(B_1(0),|x|^{\alpha})$, the compactly embedding still holds when we replace $L^p(B_1(0))$ by $L^p(B_1(0),|x|^{\alpha})$. In the following of this subsection, we focus on the subcritical case $2<p<2^*$. Let \[c(p,\alpha)=\min_{u\in H_0^1(B_1(0)),\|u\|_{p,\alpha}=1}\int_{B_1(0)} |\nabla u|^2dx.\]  
Suppose $c(p,\alpha)$ is obtained by $\tilde u\ge 0$, then $\tilde u$ satifies the following equation:
\begin{equation*}
\left\{\begin{aligned}
-\Delta \tilde u&=c(p,\alpha)|x|^{\alpha} \tilde u^{p-1},&&\text{in}\ B_1(0),\\
\tilde u&>0,&& \text{in}\ B_1(0),\\
\tilde u&=0,&& \text{on}\ \partial B_1(0).
\end{aligned}   \right. 
\end{equation*}

Next we calculate the Morse index of $\tilde u$ denoted by $m(\tilde u)$, which is the number of negative eigenvalues, counted with multiplicity of the following eigenvalue problem:
\begin{equation*}
-\Delta v-c(p,\alpha)(p-1)|x|^{\alpha} \tilde u^{p-2}v=\lambda v,v\in H_0^1(B_1(0)).
\end{equation*}

Using Rayleigh quotient, we obtain
\begin{align}\label{lambda1}
   \lambda_1&=\min_{v\in H^1_0}\frac{\int_{B_1(0)}|\nabla v|^2-c(p,\alpha)(p-1)|x|^\alpha\tilde u^{p-2}v^2dx}{\int_{B_1(0)}v^2dx} \notag \\
   &\le\frac{\int_{B_1(0)}|\nabla \tilde u|^2-c(p,\alpha)(p-1)|x|^\alpha\tilde u^{p}dx}{\int_{B_1(0)}\tilde u^2dx}\\
   &=\frac{(2-p)c(p,\alpha)\int_{B_1(0)}|x|^\alpha\tilde u^{p}dx}{\int_{B_1(0)}\tilde u^2dx}<0 \notag.
\end{align}
So $m(\tilde u)\ge 1$. Then we study the eigenvalue $\lambda_2$. Follow the ideas in \cite{Lin-MM1994}, we define \[f(t)=\frac{\int_{B_1(0)} |\nabla (\tilde u+tv)|^2dx}{\left(\int_{B_1(0)}|x|^\alpha |\tilde u+tv|^p dx\right)^{2/p}}, v\in H_0^1(B_1(0)).\]
Because $f(0)$ is a local minimum, then
\[0\le f''(0)=2\left[\int_{B_1(0)}|\nabla v|^2dx-c(p,\alpha)(p-1)\int_{B_1(0)}|x|^\alpha \tilde u^{p-2}v^2dx+c(p,\alpha)(p-2)\left(\int_{B_1(0)}|x|^\alpha \tilde u^{p-1}vdx\right)^2 \right].\]
Therefore, we have 
\[\lambda_2\ge \min_{v\perp |x|^\alpha \tilde u^{p-1}} \frac1{\int_{B_1(0)}v^2dx}\left(\int_{B_1(0)}|\nabla v|^2dx-c(p,\alpha)(p-1)\int_{B_1(0)}|x|^\alpha \tilde u^{p-2}v^2dx\right)\ge 0.\]
So the quadratic form has at most one negative
direction, which means $m(\tilde u)\le1$. Therefore $m(\tilde u)=1$. 
Let $u_{gs}=c(p,\alpha)^\frac{1}{p-2}\tilde u$, then $u_{gs}$ is a solution of \eqref{Henon}. It's easy to see that $m(u_{gs})=m(\tilde u)=1$.

\subsection{Ground state solution of supercritical case for $N\ge 3$}\label{supercriticalgroundstate}

\subsubsection*{Existence space}
Let $G:=O\left(m\right) \times O\left(n\right)$, where  $O\left(m\right),O\left(n\right)$ is the orthogonal group in  $\mathbb{R}^{m}$, $\mathbb{R}^{n}$ respectively, and $N=m+n$. Consider $H_{0, G}^1:=\left\{u \in H_0^1(B_1(0)): g u=u,\forall g \in G\right\}$, where $gu(x)=u(g^{-1}x)$. Then we can define the existence space as in \cite{Cowan-Moameni-MA2024}:
\begin{equation*}
\begin{aligned}
&K_{+}=\left\{0 \leq u \in H_{0, G}^1\left(B_1(0)\right): u
\text { is even in } \theta \operatorname{across} \theta=\frac{\pi}{4} \text { with } u_\theta \geq 0 \text { for } 0<r<1 \text { and } 0<\theta<\frac{\pi}{4}\right\} ,
\end{aligned}
\end{equation*}
where $\tilde u(s,t)=u(x)=\hat{u}(r,\theta), s=|(x_1,x_2,\cdots,x_m)|,t=|(x_{m+1},x_{m+2},\cdots,x_N)|, r=\sqrt{s^2+t^2},\theta=\arctan \frac{t}{s}$. Then \begin{align*}
 \Delta u=\sum_{i=1}^{N}u_{x_ix_i}=&\tilde u_{ss}+\frac{m-1}{s}\tilde u_s+\tilde u_{tt}+\frac{n-1}{t}\tilde u_t\\
 =&\hat u_{rr}+\frac{N-1}{r}\hat u_r+\frac{\Delta_{\mathbb{S}^{N-1}}\hat u}{r^2} \\ 
 =&\hat u_{rr}+\frac{N-1}{r}\hat u_r+\frac{\hat u_{\theta\theta}+[(n-1)\cot \theta-(m-1)\tan \theta]\hat u_\theta}{r^2}.
\end{align*}

Let $Y_4(\theta)=-\cos4\theta+\frac{2-N}{2+N}, \theta\in (0,\frac{\pi}{2})$, then when $m=n$, $-\Delta_{\mathbb{S}^{N-1}}Y_4=4(N+2)Y_4$, which  means $Y_4$ is a spherical harmonic function of degree $4$, see \cite{Cowan-Moameni-MA2024}. 
The following lemma is very useful to prove existence.
\begin{lemma}
    $K_+$ is compactly embedded in $L^p(B_1(0),|x|^{\alpha})$ when $\alpha>0, 2<p<p_\alpha$.
\end{lemma}
When $N$ is even, this lemma was proved in \cite{Cowan-Moameni-MA2024}. However, their method can be easily generated to all $N\ge3$. Here we give a different proof using the idea from \cite{Badiale-Serra-ANS2004} which is inspired by \cite{Li-JDE1990,Lin-JDE1993}.
\begin{proof}
    Let $\Omega_j=\{x:\rho^{j+1}<|x|\le \rho^j\}, \rho\in(0,1)$, then $B_1(0)\backslash\{0\}=\bigcup_{j=0}^{\infty}\Omega_j$.  Define $A={\{x:\frac{\pi}{8}\le \theta\le \frac{\pi}{4} \}}$. We only need to prove \[\left(\int_A |x|^\alpha u^pdx\right)^\frac{1}{p}\le C \left(\int_{B_1(0)}|\nabla u|^2dx\right)^\frac{1}{2}\] because of monotonicity and evenness across $\theta=\frac{\pi}{4}$.
By the two-dimensional Sobolev embedding theorem, we have \[\int_{\Omega_0\cap A}\tilde u^p(s,t)dsdt\le C\left(\int_{\Omega_0\cap A}\tilde u_s^2+\tilde u_t^2+\tilde u^2dsdt\right)^{\frac p2},\forall p<\infty.\] Here abuse of the notation is used when we don't distinguish the domain in different coordinates.

Using change of variable, we obtain
\[\int_{\Omega_j\cap A}\tilde u^p(s,t)dsdt\le C\rho^{2j}\left(\int_{\Omega_j\cap A}\tilde u_s^2+\tilde u_t^2+ \frac{1}{\rho^{2j}}\tilde u^2 dsdt\right)^\frac{p}{2}\le C\rho^{2j}\left(\int_{\Omega_j\cap A}\tilde u_s^2+\tilde u_t^2+\frac{1}{s^2+t^2}
\tilde u^2dsdt\right)^{\frac{p}{2}},\forall p<\infty.\] Therefore, 
\begin{align*}
  &\int_{B_1(0)}|x|^\alpha u^pdx 
  =\sum_{j=0}^\infty\int_{\Omega_j}|x|^\alpha u^pdx 
  \le C\sum_{j=0}^\infty \int_{\Omega_j\cap A}|x|^\alpha u^pdx\\
  = &C\sum_{j=0}^\infty \int_{\Omega_j\cap A}|s^2+t^2|^\frac{\alpha}{2} \tilde u^ps^{m-1}t^{n-1}dsdt
  \le C\sum_{j=0}^\infty\rho^{j(\alpha+N-2)} \int_{\Omega_j\cap A}\tilde u^pdsdt\\
  \le &C\sum_{j=0}^\infty\rho^{j(\alpha+N)}\left(\int_{\Omega_j\cap A}\tilde u_s^2+\tilde u_t^2+\frac{1}{s^2+t^2}
\tilde u^2dsdt\right)^{\frac{p}{2}}\\
  \le &C\sum_{j=0}^\infty\rho^{j\left[\alpha+N-\frac{p}{2}(N-2)\right]} 
  \left(\int_{\Omega_j\cap A}\left(\tilde u_s^2+\tilde u_t^2+\frac{1}{s^2+t^2}
\tilde u^2\right)s^{m-1}t^{n-1}dsdt\right)^\frac{p}{2}\\
  \le &C\sum_{j=0}^\infty\rho^{j\left[\alpha+N-\frac{p}{2}(N-2)\right]} 
  \left(\int_{B_1(0)}|\nabla u|^2+\frac{u^2}{|x|^2}dx\right)^\frac{p}{2}.
\end{align*}
By hardy inequality, we obtain when $\alpha+N-\frac{p}{2}(N-2)>0$, which means $p<p_\alpha$, we can prove 
\[\int_{B_1(0)}|x|^\alpha u^pdx \le C \left(\int_{B_1(0)}|\nabla u|^2dx\right)^\frac{p}{2}.\] Then the compact embedding can be obtained using the standard interpolation method.
\end{proof}

Let \[\tilde c(p,\alpha)=\min_{u\in K_+,\|u\|_{p,\alpha}=1}\int_{B_1(0)} |\nabla u|^2dx.\]  
Suppose $\tilde c(p,\alpha)$ is obtained by $\Tilde u$. Define \[f(t)=\frac{\int_{B_1(0)} |\nabla ((1-t)\tilde u+tv)|^2dx}{\left(\int_{B_1(0)}|x|^\alpha ((1-t)\tilde u+tv)^p dx\right)^{2/p}}, t\in (0,1),v\in K_+,\]
then we have 
\begin{equation}\label{functional inequality}
 f'_+(0)=2\left[\int_{B_1(0)}\nabla\tilde u\nabla(v-\tilde u)dx-\tilde c(p,\alpha)\int_{B_1(0)}|x|^{\alpha}\tilde u ^{p-1}(v-\tilde u)dx\right]\ge 0.   
\end{equation}
When $N\ge 4$, and $N$ is even. By the construction in the proof of Proposition 5.2 in \cite{Cowan-Moameni-MA2024}, the Dirichlet problem
\[
-\Delta w=\tilde c(p,\alpha)|x|^\alpha\tilde u^{p-1}
\quad\text{in }B_1,\qquad w=0\quad\text{on }\partial B_1
\]
has a unique weak solution $w\in K_+$. As a result,
\begin{equation}\label{pointwiseinvariance}
\int_{B_1(0)}\nabla  w\nabla(v-\tilde u)dx=\int_{B_1(0)}\tilde c(p,\alpha)|x|^{\alpha}\tilde u^{p-1}(v-\tilde u)dx.    
\end{equation}
Let $v=w$, combine \eqref{functional inequality} and \eqref{pointwiseinvariance} together, we obtain
\[\int_{B_1(0)}\nabla  (\tilde u-w)\nabla(w-\tilde u)dx\ge 0,\tilde u,w\in K_+,\]
which means $w=\tilde u$ and 
\begin{equation*}
\left\{\begin{aligned}
-\Delta \tilde u&=\tilde c(p,\alpha)|x|^{\alpha} \tilde u^{p-1},&&\text{in}\ B_1(0),\\
\tilde u&>0,&& \text{in}\ B_1(0),\\
\tilde u&=0,&& \text{on}\ \partial B_1(0),
\end{aligned}   \right. 
\end{equation*}
 Let $u_{K_+}=\tilde c(p,\alpha)^\frac{1}{p-2}\tilde u$, then $u_{K_+}$ is a solution of H\'enon equation \eqref{Henon}.

Although we prove the existence of solutions in $K_+$ in even dimensions, we cannot calculate the Morse index $m(u_{K_+})$ as in the subcritical case. However, we can compare the energy of $u_{K_+}$ with $u_R$ to obtain a multiplicity result. This will be done in the next section when all the tools are ready.

\section{Multiplicity}\label{radial}

Let $H_R=H_{0,rad}^1(B_1(0))$ and $u_R$ denote the unique positive radial solution of \eqref{Henon}. We will calculate the Morse index $m(u_R)$ in this section, and then prove theorem \ref{mainresultmultiplicity}.
\subsection{Lower bounds for the radial Morse index}
Linearization at $u_R$ gives
\begin{equation}\label{linearized}
-\Delta v-(p-1)|x|^{\alpha}  u_R^{p-2}v=\Lambda v,v\in H_0^1(B_1(0)).
\end{equation}
Expanding $v$ into spherical harmonics $v=\sum R_k(r)Y_k$ reduces the equation to 
\begin{equation}\label{linearizedradial}
\begin{cases}
 \mathcal{L}_k(R_k)=-(R_k''+\frac{N-1}{r}R_k')+\frac{\mu_k}{r^2}R_k-(p-1)r^{\alpha}  u_R^{p-2}R_k=\Lambda_{i,k} R_k, 0<r<1\\  R'_0(0)=R_0(1)=0,R_k(0)=R_k(1)=0,k\ge 1,R_k\in H^1((0,1),r^{N-1}dr),
\end{cases}
\end{equation}
where $\mu_k=k(k+N-2)$ satisfies $-\Delta_{\mathbb{S}^{N-1}}Y_k=\mu_kY_k, k\ge 1$, see \cite{Gladiali-Grossi-Neves-Adv2013}.

We notice that $\mathcal{L}_k$ is a self-adjoint operator in $H^1((0,1),r^{N-1}dr)$ with the corresponding boundary condition. So 
\[
\Lambda_{1,0}=\min_{z\in H^1((0,1),r^{N-1}dr)} \frac{\int_0^1\mathcal{L}_0(z)zr^{N-1}dr}{\int_0^1z^2r^{N-1}dr}
\le \frac{\int_0^1\mathcal{L}_0(u_R)u_Rr^{N-1}dr}{\int_0^1u_R^2r^{N-1}dr}
=\frac{(2-p)\int_0^1u_R^{p}r^{\alpha+N-1}dr}{\int_0^1u_R^2r^{N-1}dr} <0.  \] 

As for higher degree eigenvalues $\Lambda_{1,k},k\ge1$, we present the following result.
\begin{lemma}\label{Lambdaradial}
(i) $\Lambda_{1,k}<0, k\ge 1$ when $N\ge 2, (p-2)(\alpha+N-2k)\ge 4k$;\\
(ii) $\Lambda_{1,1}<0$ when $N\ge 3, 2^*<p<p_\alpha$.

\end{lemma}
\begin{proof}
(i) Let $z_k(r)=r^ku_R^{\frac p2}$, then $z_k(0)=z_k(1)=0,z>0,(0<r<1)$. In addition,
\begin{align*}
 \mathcal{L}_k(z_k)&=-(z_k''+\frac{N-1}{r}z_k')+\frac{\mu_k}{r^2}z_k-(p-1)r^{\alpha}  u_R^{p-2}z_k\\
 &=-r^{k-1}u_R^{\frac{p}{2}-2}\left[\frac{p}{2}\left(\frac{p}{2}-1\right)r(u_R')^2+kpu_R'u_R+\left(\frac{p}{2}-1\right)r^{\alpha+1}u_R^p\right]\\
 &=-r^{k-1}u_R^{\frac{p}{2}-2}\left[\left(\frac{p}{2}-1\right)
\left(\frac{p}{2}r(u_R')^2+r^{\alpha+1}u_R^p+(N+\alpha)u_R'u_R\right)+\left(kp-(N+\alpha)\left(\frac p2-1\right)\right)u_R'u_R\right].
\end{align*}

Define $G(r)=r^{N-1}\left[\frac{p}{2}r(u_R')^2+r^{\alpha+1}u_R^p+(N+\alpha)u_R'u_R\right]$, then 
\[G(0)=0,G'(r)=r^{N-1}(u_R')^2\left[N+\alpha-\frac{p}{2}(N-2)\right]>0, \text{ when } p<p_\alpha.\]
Therefore, $G(r)>0$, combining with $u_R'<0$ implies 
$\mathcal{L}_k(z_k)<0$ as long as $kp-(N+\alpha)\left(\frac p2-1\right)\le 0$, that is \[(p-2)(\alpha+N-2k)\ge4k.\]
As a result, \[\Lambda_{1,k}\le\frac{\int_0^1\mathcal{L}_k(z_k)z_kr^{N-1}dr}{\int_0^1z_k^2r^{N-1}dr}<0.\]
$\int_0^1\mathcal{L}_k(z_k)z_kr^{N-1}dr$ is well defined because $u_R=O(1-r)$ near $r=1$.



(ii) Let $z(r)=(N-2)ru_R+(r^2-1)u_R'$, then $z(0)=z(1)=0,z>0,(0<r<1)$. \\
When $k=1,\mu_k=N-1$, so
\begin{align*}
 \mathcal{L}_1(z)&=-(z''+\frac{N-1}{r}z')+\frac{N-1}{r^2}z-(p-1)r^{\alpha}  u_R^{p-2}z\\
 &=r^{\alpha-1}u_r^{p-1}\left[-\alpha+(\alpha+4-(N-2)(p-2))r^2\right]
\end{align*}
Because 
$2^*<p<p_\alpha$, so $4<(N-2)(p-2)<2\alpha+4$. As a consequence, \[-\alpha+\left(\alpha+4-(N-2)(p-2)\right)r^2<0,\ \forall r\in(0,1).\] From Rayleigh quotient, we have
\[\Lambda_{1,1}\le\frac{\int_0^1\mathcal{L}(z)zr^{N-1}dr}{\int_0^1z^2r^{N-1}dr}<0\]
\end{proof}
\begin{remark}
By (ii), we immediately obtain $m(u_R)\ge N+1$ in the supercritical case. As long as we find some other solution with low morse index, like the ground state solution $u_{K_+}$, we can obtain the multiplicity. However, we don't know how to calculate $m(u_{K_+})$. We choose to compare the energy instead.
\end{remark}

\subsection{Proof of theorem \ref{mainresultmultiplicity}}
In this subsection we will prove multiplicity results by comparing the morse index in subcritical case and by comparing energy in supercritical case.
\begin{proof}
\begin{itemize}
    \item Subcritical case $2<p<2^*, N\ge 2,(p-2)(\alpha+N-2)\ge 4$
\end{itemize}
By lemma \eqref{Lambdaradial}, $\Lambda_{1,1}<0$, which implies $m(u_R)\ge N+1$. Meanwhile, in subsection \ref{m(ugs)}, we prove ground state solution of the H\'enon equation \eqref{Henon} has morse index $m(u_{gs})=1$. So $u_{gs}\neq u_R$, which means H\'enon equation has at least two solutions, the radial solution $u_R$ and the nonradial solution $u_{gs}$. 

\begin{itemize}
    \item Supercritical case $2^*<p<p_\alpha, N\ge 4, N\text{ is even },(p-2)(\alpha+N-8)\ge 16$, $N$ is even.
\end{itemize}
We argue by contradiction. Assume $u_R=u_{K_+}$. Let \[A(v)=\int_{B_1(0)}|\nabla v|^2dx,B(v)=\int_{B_1(0)}|x|^\alpha |v|^pdx, S(v)=\frac{A(v)}{B(v)^\frac{2}{p}},\] then $\tilde u=\tilde c(p,\alpha)^{\frac{-1}{p-2}}u_R$ satisfies \[A(\tilde u)=\int_{B_1(0)}|\nabla \tilde u|^2dx=\tilde c(p,\alpha),B(\tilde u)=\int_{B_1(0)}|x|^{\alpha}u^{p}dx=1, -\Delta \tilde u=\tilde c(p,\alpha)|x|^\alpha\tilde u^{p-1}.\]
Let
\[h(s)=S(\tilde u+s\Phi)=\frac{A(\tilde u+s\Phi)}{B(\tilde u+s\Phi)^{\frac2p}},\Phi=z_4Y_4=r^4\left(\tilde c(p,\alpha)^{\frac{1}{p-2}}\tilde u\right)^{\frac{p}2}Y_4,s>0.\]
Then 
$\tilde u+s\Phi\in K_+$ when $s>0$ is small, because $\tilde u$ is radial, $Y_4$ is increasing in $(0,\frac\pi4)$, even across $\theta=\frac\pi4$ and $\left|\frac{\Phi}{\tilde u}\right|\le C $. In addition,
\begin{align*}
h(0)=&\tilde c(p,\alpha),  h'_+(0)=2\left(\int_{B_1(0)}\nabla \tilde u\nabla \Phi-\tilde c(p,\alpha)|x|^\alpha \tilde u^{p-1} \Phi dx\right)=0,\\  
 h''_+(0)=&2\left[\int_{B_1(0)}|\nabla \Phi|^2dx-\tilde c(p,\alpha)(p-1)\int_{B_1(0)}|x|^\alpha \tilde u^{p-2}\Phi^2dx+\tilde c(p,\alpha)(p-2)\left(\int_{B_1(0)}|x|^\alpha \tilde u^{p-1}\Phi dx\right)^2 \right]\\
 =&2\left(\int_0^1 \left(-z_4''-\frac{N-1}{r}z_4'+\frac{\mu_4}{r^2}z_4-\tilde c(p,\alpha)(p-1)r^\alpha \tilde u^{p-2}z_4\right)z_4r^{N-1}dr \right)\left(\int_{\mathbb{S}^{N-1}}Y_4^2d\omega\right)\\
 =&2\left(\int_0^1 \mathcal{L}_4(z_4)z_4r^{N-1}dr\right)\left(\int_{\mathbb{S}^{N-1}}Y_4^2d\omega\right)<0.   
\end{align*}
In the above calculation, we used 
\begin{align*}
 \int_{B_1(0)}|x|^\alpha \tilde u^{p-1}\Phi dx=&\int_0^1z_4\tilde u^{p-1}r^{\alpha+N-1}dr\int_{\mathbb{S}^{N-1}}Y_4dw\\
 =&\int_0^1z_4\tilde u^{p-1}r^{\alpha+N-1}dr\int_0^{\frac{\pi}{2}}\cos^{m-1}\theta\sin ^{n-1}\theta\left(-\cos4\theta+\frac{2-N}{2+N}\right)d\theta=0,  
\end{align*}
when $m=n$. As a result, $h(s)<h(0)$ when $s>0$ is small, which means $
S(\tilde u+s\Phi)<S(\tilde u)=\tilde c(p,\alpha)$, a contradiction. Therefore, $u_R\neq u_{K_+}$, resulting to a multiplicity result. 
\end{proof}

\section{Uniqueness}\label{uniqueness}
In this section, we establish the one-dimensional uniqueness result. Suppose H\'enon equation has two solutions, one is the radial (even) solution $u_R$, the other is the nonradial (noneven) solution $u$. Define $w(x)=u(x)-u(-x),x\in(-1,1)$. Then $w(x)\not \equiv0$ satisfies
\begin{equation}
\left\{\begin{aligned}
-w''(x)&=|x|^{\alpha}\left( u^{p-1}(x)-u^{p-1}(-x)\right),x\in(-1,1),\\
w(0)&=w(1)=0
\end{aligned}   \right. 
\end{equation}
Let \[h(x)=\begin{cases}
 \frac{u^{p-1}(x)-u^{p-1}(-x)}{w(x)},&w(x)\neq 0 ,\\
0 ,&w(x)=0,
\end{cases}\]
then $h(x)\le (p-1)M^{p-2}$, where $M=\|u\|_\infty$. And \[\int_0^1 |w'|^2dx=\int_0^1|x|^\alpha w^2 h(x)dx<\max h(x)\int_0^1|x|^\alpha w^2dx.\]

We introduce the weighted eigenvalue problem for $\varphi\in H_0^1((0,1))$ as follows:
\begin{equation}\label{weightedeigenvaluehalf}
\left\{\begin{aligned}
- \varphi'' &=\lambda|x|^\alpha \varphi,\ x\in(0,1),\\
\varphi(0)&=\varphi(1)=0
\end{aligned}   \right. 
\end{equation}
as a result, 
\[
\lambda_1=\inf_{\varphi\in H_0^1((0,1))} \frac{\int_{0}^1|\varphi'|^2dx}{\int_{0}^1|x|^\alpha \varphi^2dx}  
\le \frac{\int_{0}^1|w'|^2dx}{\int_0^1|x|^\alpha w^2dx}
< \max h(x)\le (p-1)M^{p-2}.\]

Next, we introduce the weighted eigenvalue problem  for even function $\psi\in H_0^1((-1,1))$ as follows:
\begin{equation}\label{weightedeigenvalue}
\left\{\begin{aligned}
- \psi'' &=\lambda|x|^\alpha \psi,x\in(-1,1),\\
\psi(-1)&=\psi(1)=0,
\end{aligned}   \right. 
\end{equation}

Let $\varphi(x)=x^{\frac12}Z(s),s=cx^{\beta}$, where $c=\frac{2\sqrt{\lambda}}{\alpha+2}, \beta=\frac{\alpha+2}{2}$. Then 
$Z$ satisfies the following Bessel equation:
\[s^2Z''(s)+sZ'(s)+(s^2-\nu^2)Z(s)=0, \nu=\frac{1}{\alpha+2}.\]
So $\varphi_n(x)=\sqrt xJ_{\nu}\left(j_{\nu,n}x^\frac{\alpha+2}{2}\right)$, and the corresponding eigenvalue is $\lambda_{n}=\left(\frac{\alpha+2}{2}\right)^2j^2_{\nu,n}$, where $j_{\nu,n}$ denote the nth positive zero of the Bessel function $J_{\nu}$. Similarly, 
$\psi_n(x)=\sqrt {|x|}J_{-\nu}\left(j_{-\nu,n}|x|^\frac{\alpha+2}{2}\right)$, and the corresponding eigenvalue is $\lambda_{n}^{even}=\left(\frac{\alpha+2}{2}\right)^2j^2_{-\nu,n}$. 

Let $v=\frac{u}{M}$, then $v$ is a solution of \begin{equation}\label{v}
\left\{\begin{aligned}
- v''&=M^{p-2}|x|^{\alpha} v^{p-1},v>0,\text{ in } (-1,1),\\
v(-1)&=v(1)=0
\end{aligned}   \right. 
\end{equation}

Combine \eqref{weightedeigenvalue} when $\lambda=\lambda_{1}^{even}$ and \eqref{v} together, we obtain
\[\lambda_{1}^{even}\int_{-1}^1|x|^\alpha\psi_{1} vdx=M^{p-2}\int_{-1}^1|x|^\alpha\psi_{1} v^{p-1}dx.\]
As a consequence, we get the following neccessary condition for H\'enon equation to have multiple solutions:
\begin{equation}\label{neccessary}
\lambda_1< (p-1)M^{p-2}=(p-1)\lambda_{1}^{even}\frac{\int_{-1}^1|x|^\alpha\psi_{1} vdx}{\int_{-1}^1|x|^\alpha\psi_{1} v^{p-1}dx}.    
\end{equation}
Let $C_\alpha=\frac{\int_{-1}^1|x|^\alpha\psi_{1} vdx}{\int_{-1}^1|x|^\alpha\psi_{1} v^{p-1}dx}$.
If \eqref{neccessary} doesn't hold, we get a uniqueness result. So the sufficient condition of uniqueness highly depends on the estimate of constant $C_\alpha$. 

By concavity of $v(x)$, we obtain $v(x)\ge \frac{1-|x|}{2}, |x|<1$. Let $A=\int_{-1}^1|x|^\alpha\psi_{1} vdx,B=\int_{-1}^1|x|^\alpha\psi_{1} v^{p-1}dx$. Then \[0\le A-B=\int_{-1}^1|x|^\alpha\psi_{1} \left(v-v^{p-1}\right)dx\le c_p\int_{-1}^1|x|^\alpha\psi_{1}dx, c_p=(p-2)(p-1)^{-\frac{p-1}{p-2}}.\]

Therefore, 
\begin{align*}
    1-\frac{B}{A}\le &c_p\frac{\int_{-1}^1|x|^\alpha\psi_{1}dx}{\int_{-1}^1|x|^\alpha\psi_{1}vdx}
\le c_p\frac{\int_{-1}^1|x|^\alpha\psi_{1}dx}{\int_{-1}^1|x|^\alpha\psi_{1} \frac{1-|x|}{2}dx } 
=2c_p\frac{\int_{0}^1|x|^\alpha\psi_{1}dx}{\int_{0}^1x^\alpha\psi_{1}dx-\int_{0}^1x^{\alpha+1}\psi_{1}dx }\\
=&2c_p\frac{-\int_{0}^1\psi_{1}''dx}{\left(-\int_{0}^1\psi_{1}''dx+\int_{0}^1x\psi_{1}''dx\right)}
=2c_p\frac{-\psi_{1}'(1)}{\psi_{1}(0)}
=c_p\frac{(\alpha+2)J_{-\nu+1}\left(j_{-\nu,1}\right)j_{-\nu,1}}{\left(j_{-\nu,1}/2\right)^{-\nu}\frac1{\Gamma\left(1-\nu\right)}}\\
=&c_p(\alpha+2)2^{-\nu}J_{-\nu+1}\left(j_{-\nu,1}\right)\left(j_{-\nu,1}\right)^{1+\nu}\Gamma\left(1-\nu\right).
\end{align*}
Let $D_\alpha=(\alpha+2)2^{-\nu}J_{-\nu+1}\left(j_{-\nu,1}\right)\left(j_{-\nu,1}\right)^{1+\nu}\Gamma\left(1-\nu\right)$. Then \[0\le c_pD_\alpha< 1,  C_\alpha=\frac{A}{B}\le \frac{1}{1-c_pD_\alpha}.\]
This implies that if we assume 
\[\lambda_{1}\ge(p-1)\lambda_{1}^{even}\frac{1}{1-c_pD_\alpha},\]
or 
\[(p-1)\frac{\lambda_{1}^{even}}{\lambda_{1}}+c_pD_\alpha\le 1,\]
then the uniqueness holds. To be more specific, if  
\begin{equation}\label{uniqueness1d}
(p-1)\left(\frac{j_{-\nu,1}}{j_{\nu,1}}\right)^2+(p-2)(p-1)^{-\frac{p-1}{p-2}}(\alpha+2)2^{-\nu}J_{-\nu+1}\left(j_{-\nu,1}\right)\left(j_{-\nu,1}\right)^{1+\nu}\Gamma\left(1-\nu\right)\le 1,    
\end{equation}
then the 1 dimensional H\'enon equation has a unique solution.

Next, we prove the Tanaka's uniqueness region in \cite{Tanaka-JDE2013} is contained in our uniqueness region.
\begin{proposition}
Let $D=\{(p,\alpha):p>2,\alpha>0,(p-1)2^{p-1}[B(\alpha+1,p+1)]^{-1}\le \pi^2\}$, then for any $(p,\alpha) \in D$, $(p,\alpha)$ satisfies  \eqref{uniqueness1d}.  
\end{proposition}\label{comparationwithTanaka}
\begin{proof}
By $(p-1)2^{p-1}[B(\alpha+1,p+1)]^{-1}\le \pi^2$, we have $(p-1)2^{p-1}\le B(\alpha+1,p+1)\pi^2\le \frac{\pi^2}{p+1}$, therefore, when $(p,\alpha)\in D$, $p<2.262$.
In order to prove \eqref{uniqueness1d} holds when $(p,\alpha)\in D$, we only need to estimate $\lambda_{1},\lambda_{1}^{even},D_{\alpha}$.

\begin{align*}
\lambda_{1}^{even} &=\inf_{
\substack{R_0\in H^1_0(-1,1),\\R_0 even} } \frac{\int_{-1}^1|R_0'|^2dx}{\int_{-1}^1 R_0^2|x|^{\alpha}dx}
\le\frac{\int_{-1}^1\left|(1-|x|)'\right|^2dx}{\int_{-1}^1 \left(1-|x|\right)^2|x|^{\alpha}dx}
=\frac{1}{B(\alpha+1,3)},\\
\lambda_{1}& =\inf_{R_0\in H^1_0(0,1) } \frac{\int_{0}^1|R_0'|^2dx}{\int_{0}^1 R_0^2|x|^{\alpha}dx}
\ge \inf_{R_0\in H^1_0(0,1) } \frac{\int_{0}^1|R_0'|^2dx}{\int_{0}^1 R_0^2dx}=\pi^2,\\
D_\alpha&= \frac{\int_{-1}^1|x|^\alpha\psi_{1}dx}{\int_{-1}^1|x|^\alpha\psi_{1} \frac{1-|x|}{2}dx }
\le \frac{\int_{-1}^1|x|^\alpha dx}{\int_{-1}^1|x|^\alpha(1-|x|) \frac{1-|x|}{2}dx }
=\frac{2}{(\alpha+1)B(\alpha+1,3)}.
\end{align*}
Here we use the concavity of $\psi_1$. Therefore,
\begin{align*}
&(p-1)\frac{\lambda_{1}^{even}}{\lambda_{1}}+c_pD_\alpha\\
&\le   \frac{p-1}{\pi^2B(\alpha+1,3)}+ (p-2)(p-1)^{-\frac{p-1}{p-2}}\frac{2}{(\alpha+1)B(\alpha+1,3)}\\
&\le\frac{1}{B(\alpha+1,p+1)}\left(\frac{p-1}{\pi^2}+(p-2)(p-1)^{-\frac{p-1}{p-2}}\frac{2}{\alpha+1}\right)\\
&\le\frac{\pi^2}{(p-1)2^{p-1}}\left(\frac{p-1}{\pi^2}+2(p-2)(p-1)^{-\frac{p-1}{p-2}}\right).
\end{align*}
Let $g(p)=\frac{\pi^2}{(p-1)2^{p-1}}\left(\frac{p-1}{\pi^2}+2(p-2)(p-1)^{-\frac{p-1}{p-2}}\right), 2<p<2.262$, then 
$g(p)$ is increasing in $(2,2.262)$, which implies 
$g(p)<g(2.262)\approx 0.974 <1$.
As a consequence, when $(p,\alpha)\in D$, \eqref{uniqueness1d} holds.
\end{proof}

The estimates above provide explicit sufficient regions for symmetry breaking and one-dimensional uniqueness. They
do not locate the exact transition between uniqueness and multiplicity. Determining sharper boundaries, extending
the cone argument to odd dimensions, and obtaining uniform quantitative uniqueness estimates in higher dimensions
remain natural questions.

\section*{Acknowledgments}
This work was supported by the Fundamental Research Funds for the Central Universities (2025), Grant JUSRP202501053, and the “Taihu Light” Science and Technology Tackling Program (Basic Research) Project, Grant 20260126KYY0001. I wish to thank Liping Wang (East China Normal University) for introducing this problem in the seminar, Shulin Zhou (Peking University) for encouraging the study of Morse indices, and Yaoting Gui (Xiamen
University) for helpful discussions.


\section*{Data availability}
This paper has no associated data.

\section*{Conflicts of Interest}
The authors declare that there is no conflict of interest.

\bibliography{henonbib}
\end{document}